\documentclass[12pt,reqno]{amsart}
\usepackage[centertags]{amsmath}
\usepackage{mathrsfs}
\usepackage{setspace}

\theoremstyle{plain} 

\newtheorem{lem}{Lemma}
\newtheorem{pro}{Proposition}
\newtheorem*{conj}{Conjecture}
\newtheorem{thm}{Theorem}

\theoremstyle{definition}

\theoremstyle{remark}

\numberwithin{equation}{section}

\newcommand{\set}[1]{\left\{#1\right\}}
\newcommand{\pd}{\,\partial}
\newcommand{\comment}[1]{}

\begin{document}

\title[]{Some comparison results on  equivalence groups}
\author[J C Ndogmo]{J C Ndogmo}

\address[J C Ndogmo]{School of Mathematics\\
University of the
Witwatersrand\\
Private Bag 3, Wits 2050\\
South Africa}
\email{jean-claude.ndogmo@wits.ac.za}

\begin{abstract} This paper deals with the comparison of two
common types of equivalence groups of differential equations, and this
gives rise to a number of results presented in the form
of theorems. It is shown in particular that one type can be identified with a subgroup of the other type. Consequences of this comparison related in particular
to the determination of invariant functions of the differential
equations are also discussed.
\end{abstract}
%




\maketitle 


\section{Introduction}
\label{s:intro}

An invertible point transformation that maps every element in a
family ${\mathcal{F}}$ of differential equations of a specified form
into the same family is commonly referred to as an equivalence
transformation of the equation \cite{ovsi-bk, ibra-lap, waf-joint}.
Elements of the family ${\mathcal F}$ are generally labeled by a set
of arbitrary functions, and the set of all equivalence
transformations forms, in general, an infinite dimensional Lie group
called the equivalence group of $\mathcal{F}.$ One type of
equivalence transformations usually considered
\cite{ovsi-bk,ibratoris, zedan} is that in which the arbitrary
functions are also transformed. More specifically, if we denote by
$A= (A_1,\dots, A_m)$ the arbitrary functions specifying the family
element in $\mathcal{F},$ then for given independent variables
$x=(x^1, \dots, x^p)$ and dependent variable $y,$ this type of
equivalence transformations takes the form 
\begin{subequations} \label{eq:eqv1}
\begin{align}
x &= \varphi (z, w, B) \\
y &= \psi (z, w, B) \\
A &= \zeta (z, w, B)
\end{align}
\end{subequations}
where $z= (z^1, \dots, z^p)$ is the new set of independent
variables, $w~=~w(z),$ is the new dependent variable, and $B= (B_1,
\dots, B_m)$ represents the new set of arbitrary  functions. The
original  arbitrary functions $A_i$ may be functions of $x,y,$ and
the derivatives of $y$ up to a certain order, although quite often
they arise naturally as functions of $x$ alone, and for the equivalence
transformations of the type \eqref{eq:eqv1},  the corresponding
equivalence group that we denote by $G_S$ is simply the symmetry
pseudo-group of the equation, in which the arbitrary functions  are
also considered as additional dependent variables.\par

 The other type of equivalence transformations
commonly considered \cite{forsyth, ibra-lap, ndogftc, schw-pap}
involves only the ordinary variables of the equation, i.e. the
independent and the dependent variables, and thus with the notation
already introduced, it consists of point transformations of the form
\begin{subequations} \label{eq:eqv2}
\begin{align}
x &= \varphi (z, w) \\
y &= \psi (z, w).
\end{align}
\end{subequations}
If we let $G$ denote the resulting equivalence group, then it
follows from a result of Lie \cite{lieGc}, that $G$ induces another
group of transformations $G_c$ acting on the arbitrary functions of
the equation. The invariants of the group $G_c$ are what are
referred to as the invariants of the family $\mathcal{F}$ of
differential equations. These invariant functions play a crucial
role in the classification and integrability of differential
equations \cite{forsyth, lag, brio,liouv,mrz-clhyper, ovsi-bk,
schw-bk}.\par

In the recent scientific literature, there has been a great deal of
interest for finding infinitesimal methods for the determination of
invariant functions of differential equations \cite{ibra-not,
ndogpla, ndogftc, melesh, ibra-lap}. Some of these methods consist
in finding the infinitesimal (generic) generator $X$ of $G_S,$ and
then use it in one way or another \cite{ndogftc,ibra-par} to obtain
the infinitesimal generator $X^0$ of $G_c,$ which gives the
determining equations for the invariant  functions. Most of these
methods remain computationally demanding and in some cases quite
inefficient, perhaps just because the connection between the three
groups $G, G_c$ and $G_S$ does not seem to have been fully
investigated.\par

We therefore present in this
paper a comparison of the groups $G$ and $G_S$  and show in
particular that $G$ can be identified with a subgroup of $G_S,$ and
we exhibit a case where the two groups are isomorphic.  As an
extension of some results obtained in a recent paper \cite{ndogftc} on the
relationship between the generators of $G_c$ and $G_S$ when $A=
A(x),$ we show that the generator $X$ of $G_S$ admits a simple
linear decomposition of the form $X= X^1 + X^2,$ where $X^1$ is just
the generator $X^0$ of $G_c$ in which one more component has been
added, and for a certain vector field $X^2,$ and we also give a very
simple and systematic method for extracting $X^1$ from $X.$ This
decomposition also turns out to be intimately associated with the
Lie algebraic structure of the equation, as we show that $X^1$ and
$X^2$ each generate  a Lie algebra, the two of which are closely
related to the components of the Levi decomposition of the Lie
algebra of $G_S.$ This decomposition results for $X$ which we prove
for a specific linear equation are stated as a conjecture for a
general linear equation.


\section{The relationship between $G$ and $G_S$}

We shall call Type I the equivalence transformations of the form
\eqref{eq:eqv2}, and Type II those of the form \eqref{eq:eqv1},
whose equivalence groups we have denoted by $G$ and $G_S,$
respectively. When the coordinates system in which a vector field is
expressed is clearly understood, it shall be represented only by its
components, so that  a vector field 
\begin{equation*}
\omega = \xi \pd_x + \eta \pd_y + \phi \pd_A
\end{equation*}
shall be represented simply by $\omega = \set{\xi, \eta, \phi}.$ On
the other hand, for a vector $a=(a^1,\dots, a^n)$ representing a
subset of coordinates, the notation $f \pd_a$ shall mean 
\begin{align*}
f \pd_a &= f^1 \pd_{a^1} +\dots + f^{\,n} \pd_{a^n},\\[-5mm]
\intertext{and \vspace{-5mm} }
f&= (f^1, \dots, f^{\,n}).
\end{align*}
Hence with the notation introduced in the previous section, we may
represent the generator $X$ of $G_S$ as
$$
X= \set{\xi, \eta, \phi} \equiv \xi \pd_x + \eta \pd_y + \phi \pd_A.
$$
Let $V= \set{\xi, \eta}$ be the projection of this generator into
the $(x,y)$-space, and $V^0= \set{\xi^0, \eta^{\,0}}$ the
infinitesimal generator of $G.$ Elements of $\mathcal{F}$ may be
thought of as differential equations of the form
\begin{equation}\label{eq:main}
\Delta(x, y_{(n)}; A_1, \dots, A_m)=0,
\end{equation}
where $y_{(n)}$ denotes $y$ and all its derivatives up to the order
$n.$ We have the following result.
\begin{thm}\mbox{$\quad$} \label{th:GinGS}
\begin{enumerate}
\item[(a)] The  group $G$ can be identified with a
subgroup of $G_S.$
\item[(b)] The component functions $\xi^0$ and $\eta^0$ are special
instances of the functions $\xi$ and $\eta,$ respectively.
\end{enumerate}
\end{thm}
\begin{proof}
Suppose  that the action of $G_c$ induced by that of $G$ on the
arbitrary functions of the equation is given by the transformations
\begin{equation}\label{eq:Gc}
A_i= \gamma_i (z, w, B_1, \dots, B_m), \quad i=1, \dots, m.
\end{equation}
Then, since \eqref{eq:eqv2} leaves the equation invariant except for the arbitrary functions,  by also viewing the functions $A_i$ as dependent variables,
\eqref{eq:eqv2} together with \eqref{eq:Gc} constitute a symmetry
transformation of the equation. This is more easily seen if we
consider the inverse transformations of \eqref{eq:eqv2} which may be
put in the form
\begin{subequations}\label{eq:eqv2i}
\begin{align}
z &= \hat{\varphi} (x, y)\\
w &= \hat{\psi} (x,y).
\end{align}
\end{subequations}
If we now denote by
\begin{equation}\label{eq:Gci}
B_i= \hat{\gamma} (x,y, A_1, \dots, A_m), \quad i=1, \dots, m
\end{equation}
the resulting arbitrary functions in the transformed equation, it
follows that in terms of the new set of variables $z, w,$  and
$B_i,$ any element of $\mathcal{F}$ is locally invariant under
\eqref{eq:eqv2i} and \eqref{eq:Gci}, and this proves the first part
of the Theorem. The second part of the Theorem is an immediate
consequence of the first part, for we can associate with any element
$(\varphi, \psi)$ of $G$ a triplet $(\varphi, \psi, \gamma)$ in
$G_S,$ where $\gamma$ is the action in \eqref{eq:eqv2}induced by $(\varphi, \psi)$ on
the arbitrary functions of the equation. The result thus follows by
first recalling that $G_S$ has generic generator $X= \set{\xi, \eta,
\phi},$ and  by considering the infinitesimal counterpart of the
finite transformations $(\varphi, \psi, \gamma),$ which must be of
the form $\set{\xi^0, \eta^0, \zeta^0}$ for a certain function
$\zeta^0.$
\end{proof}

The second part of this theorem was differently obtained in
a recent paper \cite{ndogftc}, by observing that for any element $g= (\varphi,
\psi)$ in $G,$ if $\gamma$ denotes the corresponding induced action
on the arbitrary functions of the equation, then the set of
transformations $\set{g, \gamma^{\,-1}}$ constitutes a symmetry of
the equation in which the arbitrary functions $A= (A_1, \dots, A_m)$
are seen as additional dependent variables. However, Theorem
\ref{th:GinGS} extends this observation by showing the clear and
simple connection between the entire groups $G$ and $G_S,$ and thus gives a more
general way of obtaining symmetries of the equation from a given
element of $G.$\par

It is clear that one can get the generator $V^0= \set{\xi^0,
\eta^0}$ of $G$ by imposing on the projection $V=\set{\xi, \eta}$ of
$X$  the set of minimum conditions $\Omega$ that reduces it to the
infinitesimal generator of the equivalence group $G$ of
$\mathcal{F},$ so that $V^0= V_{\vert_{\Omega}}.$ It was also
observed (see Ref. ~\cite{ndogftc}) that in the case where $A$ is the
function of $x$ alone, if we let $\phi^0$ denote the resulting value
of $\phi$ when these minimum conditions are also imposed on
$X=\set{\xi, \eta, \phi},$ then the generator $X^0$ of $G_c$ can be
obtained by setting $X^0= \set{\xi^0, \phi^0}.$ However, the problem
that arises is that of finding the simplest and most systematic way
of extracting $\set{\xi^0, \eta^0, \phi^0}= X_{\vert_{\Omega}}$ from
$\set{\xi, \eta, \phi}.$ We shall look at the connection between
these two vectors by considering the case of a family of third order
linear ordinary differential equations (ODEs). To begin with, we
note that the coefficient $\phi^0$ is an $m$-component vector that
depends in general on $(p+1)+m$ variables, and  finding its
corresponding finite transformations by integrating the vector field
$\set{\xi^0, \eta^0, \phi^0}$ can be a very complicated task.
Fortunately, once the finite transformations of the generator $V^0$
of $G$ which are easier to find are known, we can easily obtain
those associated with $\phi^0$ using the following result.
\begin{lem}\label{le:phi0}
The finite transformations associated with the component $\phi^0$ of
$$X^1= \set{\xi^0, \eta^0, \phi^0}$$
are precisely given by the action \eqref{eq:Gci} of $G_c$ induced
by that of  \eqref{eq:eqv2i}.
\end{lem}

\begin{proof}
Since $X^1= X_{\vert_{\Omega}},$ where $\Omega$ is the set of
minimum conditions to be imposed on $V=\set{\xi, \eta}$ to reduce it
into an infinitesimal generator $V^0\!~\!=\!~\!\set{\xi^0, \eta^0}$
of $G,$ it first follows that once the finite transformations
\eqref{eq:Gci} corresponding to $V^0$ are applied to the equation,
the resulting equation is invariant, except for the expressions of
the arbitrary functions which are now given by \eqref{eq:eqv2i}.
Thus if $(z, w, b)$ are the new variables generated by the symmetry
transformations  of  $X^1,$  where $b=(b_1, \dots, b_m),$ then the
only way to have an invariant equation is to set
$$
b_i= \hat{\gamma} (x,y, A_1, \dots, A_m), \quad i=1, \dots, m
$$
where $\hat{\gamma}$ is the same function appearing in
\eqref{eq:Gci}, and this readily proves the lemma.
\end{proof}

\section{Case of the general  third order linear ODE.}
\label{s:3rdode}
Consider the general linear ODE
\begin{equation}\label{eq:e3nor}
y^{(3)} + a^1 (x) y' + a^0 (x) y=0,
\end{equation}
which is said to be in its normal reduced form. Here, the arbitrary
functions $A_i$ of the previous section are simply the coefficients
$a^j$ of the equation. This form of the equation is in no way
restricted, for  any general linear third order ODE can be
transformed into \eqref{eq:e3nor} by a simple change of the
dependent variable \cite{schw-pap, ndogpla}. If we consider the
arbitrary functions $a^j$ as additional dependent variables, then by
applying known procedures for finding Lie point symmetries
\cite{olv-bk, bluman, olv-bk2}, the infinitesimal generator $X$ of
the symmetry group $G_S$ in the coordinates system $(x, y, a^1,
a^0)$ is found to be of the form
\begin{subequations}\label{eq:X3ode}
\begin{align}
X &= \set{f, (k_1 + f')y+ g, -2 \left(a^1 f' + f^{(3)}\right), C_4}\\
\intertext{where}
C_4 &= - \frac{1}{y} \left(  a^0 g + a^1 g' + g^{(3)} \right) -
\left( 3 a^0 f' + a^1 f'' + f^{(4)}\right),
\end{align}
\end{subequations}
and where $f,$ and $g$ are arbitrary functions of $x.$ The
projection of $X$ in the $(x,y)$-space is therefore
\begin{equation}\label{eq:V3ode}
V= \set{f, (k_1 + f')y+ g},
\end{equation}
and a simple observation of this expression shows that due to the
homogeneity of \eqref{eq:e3nor}, \eqref{eq:V3ode} may represent an
infinitesimal generator of the equivalence group $G$ only if $g=0.$
A search for the one-parameter subgroup $\exp (t W),$ satisfying
$\exp (t W) (x,y)= (z,w)$ and generated by the resulting reduced
vector field $W= \set{f, (k_1+ f')y}$ readily gives
\begin{align*}
\dot{z} &= f(z)\\
\dot{w} &=  \left(  k_1 + f'(z) \right) w,\\
\intertext{where}
\dot{z} &= d z/ dt, \qquad \dot{w} = d w/ dt.
\end{align*}
Integrating these last two equations while taking into account the
initial conditions gives
\begin{subequations}\label{eq:1sgp3}
\begin{align}
J(z) &= t+ J(x) \label{eq:1sgp3a}\\
w  &= e^{k_1 t} \frac{f(z)}{f(x)} y\\
\intertext{where}
 J(z) &= \int \frac{dz}{f(z).}
\end{align}
\end{subequations}
Differentiating both sides of \eqref{eq:1sgp3a} w.r.t. $x$ shows
that $d z/ dx= f(z)/f(x).$ Thus, if we assume that $z$ is explicitly
given by
$$
z= F_t (x) \equiv F(x),
$$
for some function $F,$ then this leads to
\begin{subequations}  \label{eq:1sgp3fn}
\begin{align}
z &= F(x) \\
w &= e^{k_1 t} F'(x) y, \label{eq:1sgp3fn2}
\end{align}
\end{subequations}
and we thus recover the well-known equivalence transformation \cite{forsyth, schw-bk, schw-pap} of
\eqref{eq:e3nor}. Therefore, the
condition $g=0$ is the necessary and sufficient condition for the
vector $V$ in \eqref{eq:V3ode} to represent the infinitesimal
generator of $G,$ and hence we have
\begin{subequations}
\begin{align}\label{eq:fnlgen3}
V^0 &= \set{f, (k_1+ f')y} \equiv \set{\xi^0, \eta^0} \\
X^1 &= \set{f, (k_1+ f')y, -2 (a^1 f' + f'''), - \left( 3 a^0 f' + a^1 f'' + f^{(4)}\right)} \\
X^2 &= \set{0, g, 0, \frac{-1}{y} \left(  a^0 g + a^1 g' + g^{(3)}
\right)}, \label{eq:x2e3nor}
\end{align}
\end{subequations}
where $X^1= X_{ \vert_{g=0}},$ and $X^2= X- X^1.$ The same as we set
$X=\set{\xi, \eta, \phi},$ and $X^1= \set{\xi^0, \eta^0, \phi^0},$
for any given equation of the form \eqref{eq:main}, we also set
$X^2=\set{\xi^c, \eta^c, \phi^c},$ and in the actual case we have
\begin{align*}
\phi^0 &= \set{ -2 (a^1 f' + f'''), - \left( 3 a^0 f' + a^1 f'' +
f^{(4)}\right)}\\
\xi^c &= 0, \qquad \eta^c= g,
\intertext{ and }
\phi^c &= \set{0, \frac{-1}{y} \left(  a^0 g + a^1 g' + g^{(3)}
 \right)}. \\
\end{align*}
 Consequently, it appears that in the case of \eqref{eq:e3nor}, we
have a very simple linear relation of the form
\begin{equation}\label{eq:lin1}
\xi= \xi^0+ \xi^c, \qquad  \eta= \eta^0 + \eta^c, \qquad \phi=
\phi^0 + \phi^c,
\end{equation}
which shows that the functions $\xi^0, \eta^0,$ and $\phi^0$
representing the  components of the infinitesimal generators $V^0$
of $G$ or $X^0$ of $G_c$ depend linearly on those of $X.$ More
importantly, the set $\Omega$ of necessary and sufficient conditions
to be imposed on $X$ to obtain $X^1$ was reduced in this case to
setting $g=0.$ We would like to generalize this to a general linear
ODE.\par

The components $X^1$ and $X^2$ of $X$ that we've just exhibited also
have very specific algebraic properties. Since $X^1$ depends on $f$
and $k_1$ while $X^2$ depends on $g,$ we set
\begin{subequations}\label{eq:x1x2}
\begin{align}
X^1 (f, k_1) &= \set{f, (k_1+ f')y, -2 (a^1 f' + f'''), - \left( 3 a^0 f' + a^1 f'' + f^{(4)}\right)} \\
X^2  (g )&= \set{0, g, 0, \frac{-1}{y} \left(  a^0 g + a^1 g' +
g^{(3)} \right)},
\end{align}
\end{subequations}
for any arbitrary functions $f$ and $g$ and arbitrary constant
$k_1.$ Let $L_0, L_1$ and $L_2$ be the vector spaces generated by
$X^1(f,0), X^1(f,k_1)$ and $X^2 (g),$ respectively. Let
$$L_{S,0} = L_0 \dot{+} L_2$$
be the subspace of the Lie algebra $L_S=L_1 \dot{+} L_2$ of $G_S.$ We note that
$L_{S,0}$ is obtained from $L_S$ simply by setting $k_1=0$ in the
generator $X^1(f, k_1)$ of $G_S,$ which according to
\eqref{eq:1sgp3fn2} amounts to ignoring the constant factor $\lambda
= e^{k_1 t}$ in the transformation of the dependent variable under
$G.$  Moreover, we have $\dim L_{S,0}= \dim L_S -1,$ while $L_S$
itself is infinite dimensional in general.
\begin{thm}\label{th:levi3nor} $\mbox{$\quad$}$
\begin{enumerate}
\item[(a)] The vector spaces $L_0, L_1$ and $L_2$ are all Lie
subalgebras of $L_S.$
\item[(b)] $L_0$ and $L_2$ are the components of the Levi
decomposition of the Lie algebra $L_{S,0},$ that is
\begin{equation}\label{eq:leviLS}
 \qquad [L_0, L_2] \subset L_2,
\end{equation}
and  $L_2$ is a solvable ideal while $L_0$ is semisimple.
\end{enumerate}
\end{thm}
\begin{proof}
A computation of the commutation relations of the vector fields
shows that
\begin{subequations}\label{eq:lsbrac}
\begin{align}
\left[ X^1(f_1, k_1), X^1(f_2, k_2) \right] &= X^1 (-f_2 f_1'+ f_1 f_2', 0) \label{eq:x1brac}\\
\left[ X^2(g_1), X^2(g_2)  \right] &= 0  \label{eq:x2brac}\\
\left[ X^1(f_1, k_1), X^2(g_1)\right] &= X^2\left(f_1 g_1' - g_1(k_1
+f_1')\right),  \label{eq:x12brac}
\end{align}
\end{subequations}
where the $f_j, g_j$ are arbitrary functions, while the $k_j$ are
arbitrary constants. Consequently, it readily follows from
\eqref{eq:x1brac} and \eqref{eq:x2brac} that $L_1$ and $L_2$ are Lie
subalgebras of $L_S,$ while setting $k_1= k_2=0$ in
\eqref{eq:x1brac} shows that $L_0$ is also a Lie subalgebra, and
this proves the first part of the Theorem. Moreover, it clearly
follows from \eqref{eq:x1brac} and \eqref{eq:x12brac} that $L_{S,0}$
is an ideal in $L_S,$  while \eqref{eq:x2brac} and
\eqref{eq:x12brac} show that $L_2$ is an abelian ideal in $L_S$, and
in particular in $L_{S,0}.$ Thus we are only left with showing that
$L_0$ is an semisimple subalgebra of $L_{S,0}.$ If $L_0$ had a
proper ideal $A,$ then for a given nonzero operator $X^1(H, 0)$ in
$A,$ all operators $X^1(-f H' + f' H,0)$ would be in $A$ for all
possible functions $f.$ However, since for every function $h$ of $x$
the equation
$$
-f H' + f' H=h
$$
admits a solution in $f,$ it follows that $A$ would be equal to
$L_0.$ This contradiction shows that $L_0$ has no proper ideal, and
it is therefore a simple subalgebra of $L_{S,0}.$
\end{proof}

   Note that part (b) of Theorem \ref{th:levi3nor} can also be interpreted as stating that up to a constant
factor, $X^1$ and $X^2$ generate the components of the Levi
decomposition of $L_S.$ Although we have stated the results of this
theorem  only for the general linear third order equation
\eqref{eq:e3nor} in its normal reduced form, these results can
certainly be extended to the general linear ODE
\begin{equation}\label{eq:glinode}
y^{(n)} + a^{n-1} y^{(n-1)} + a^{n-2} y^{(n-2)}+ \dots + a^0 y=0
\end{equation}
of an arbitrary order $n \geq 3.$  We first note that if we write
the infinitesimal generator $X$ of the symmetry group $G_S$ of this
equation in the form
$$
X= \set{\xi, \eta, \phi} \equiv \xi \pd_x + \eta \pd_y + \phi \pd_A,
$$
where $A= \set{ a^{n-1}, a^{n-2}, \dots, a^0}$ is the set of all
arbitrary functions, then on account of the linearity of the
equation, we must have
\begin{equation}\label{eq:eta}
\eta = h \, y + g
\end{equation}
for some arbitrary functions $h$ and $g.$ Now, let again $X^1=
\set{\xi^0, \eta^0, \phi^0}$ and $X^2$ be given by
\begin{equation}\label{eq:x1x2}
X^1 = X_{\vert_{g=0}}, \qquad X^2= X-X^1,
\end{equation}
 and set $X^0= \set{\xi^0, \phi^0}.$  We have shown in
another recent paper \cite{ndogpla} that $X^0$ thus obtained using $g=0$  as the minimum
 set of conditions is the infinitesimal generator of the group
$G_c$ for $n=3,4,5.$ This should certainly also hold for the linear
equation  \eqref{eq:glinode} of a general order, and we thus propose
the following.
\begin{conj}
For the general linear {\textrm  ODE} \eqref{eq:glinode}, $X^0= \set{\xi,
\phi}_{\vert_{g=0}}$  is the infinitesimal generator of $G_c,$ where $X= \set{\xi, \eta, \phi}$ is the generator of $G_S.$

\comment{let $X^1$ and $X^2$ be given by \eqref{eq:x1x2}, in which
the function $g$ is given by \eqref{eq:eta}, and set $X^1=
\set{\xi^0, \eta^0, \phi^0}.$ Let $L_1$ and $L_2$ be the vector
spaces generated by $X^1$ and $X^2,$ respectively.
\begin{enumerate}
\itemsep= 1mm
\item[(a)] $L_1$ is a simple subalgebra of $L_S$ and $L_2$ is the
radical of $L_S,$ so that the sum
$$
L_S= L_1 \dot{+} L_2
$$
represents the Levi decomposition of $L_S.$
\item[(b)]
\end{enumerate}
}
\end{conj}
It has been proved \cite{ndogftc} that for any family $\mathcal{F}$ of
(linear or nonlinear) differential equations of any order in which
the arbitrary functions depend on the independent variables alone,
if $X^1= \set{\xi^0, \eta^0, \phi^0}$ is obtained by setting $X^1=
X_{\vert_\Omega}$ for some set $\Omega$ of minimum conditions that
reduce $V=\set{\xi, \eta}$ into a generator of $G,$ then $X^0=
\set{\xi^0, \phi^0}$ is the generator of $G_c.$ However, the
difficulty lies in finding the set $\Omega$ of minimum conditions,
and we have proved that for \eqref{eq:e3nor}, $\Omega$ is given by
$\set{g=0},$ and extended this as a conjecture for a general linear
homogeneous ODE. \par
Moreover, calculations done for equations of low order up to five
suggest that all subalgebras appearing in Theorem \ref{th:levi3nor}
can also be defined in a similar way for the general linear equation
\eqref{eq:glinode}, and that all the results of the theorem also
holds for this general equation.\par

We now wish to pay some attention to the converse of part (a) of
Theorem \ref{th:GinGS} which states that for any given family
$\mathcal{F}$ of differential equations, the group $G$ can be viewed
as a subgroup of $G_S.$ From the proof of that theorem it appears
that the symmetry group $G_S$ is much larger in general, because
there are symmetry transformations that do not arise from Type I
equivalence transformations. A simple example of such a symmetry is
given by the term $X^2$ appearing in \eqref{eq:x2e3nor}, of the
symmetry generator of \eqref{eq:e3nor}. Indeed, by construction its
projection $X^{2,0}= \set{0,g}$ in the $(x,y)$-space does not match
any particular form of the generic infinitesimal generator $V^0=
\set{f, (k_1 + f')y}$ of $G,$ where $f$ is an arbitrary function and
$k_1$ an arbitrary constant.\par

Nevertheless, although \eqref{eq:e3nor} gives an example in which
the inclusion $G \subset G_S$ is strict, there are equations for
which the two groups are isomorphic. Such an equation is given by
the nonhomogeneous version of \eqref{eq:e3nor} which may be put in
the form
\begin{equation}\label{eq:e3nh}
y^{(3)} + a^1 (x) y' + a^0 (x) y + r(x)=0,
\end{equation}
where $r$ is also an arbitrary function, in addition to $a^1$ and
$a^0.$ The linearity of this equation forces its equivalence
transformations to be of the form
\begin{equation}\label{eq:eqv3nh0}
x= f(z), \qquad y = h(z) w + g(z),
\end{equation}
and the latter change of variables transforms \eqref{eq:e3nh} into
an  equation of the form
$$
w''' + B_2 w'' + B_1 w' + B_0 w + B_{-1}=0,
$$
where the $B_j,\;$ for $j=-1, \dots,2$ are functions of $z$ and
$$
B_2= 3 \left( \frac{h'}{h} - \frac{f''}{f'}\right).
$$
The required vanishing of $B_2$ shows that the necessary and
sufficient condition for \eqref{eq:eqv3nh0} to represent an
equivalence transformation of \eqref{eq:e3nh} is to have $h= \lambda
f'$ for some arbitrary  constant $\lambda.$ The equivalence
transformations of \eqref{eq:e3nh} are therefore given by
\begin{equation}\label{eq:eqv3nh}
x= f(z), \qquad y= \lambda f'(z) w + g(z).
\end{equation}
On the other hand, the generator $X$ of the symmetry group $G_S$ of
the nonhomogeneous equation \eqref{eq:e3nh} in the coordinates system $(x,y, a^1, a^0, r)$ is
found to be of the form
\begin{subequations}\label{eq:xe3nh}
\begin{align}
X &= \set{J, (k_1 + J')y + P, -2 \left( a^1 J' + J'''\right), C_3,
\phi^4}\\
\intertext{ where }
C_3 &=   \frac{-1}{y} \left(  a^0 P+\phi^4+ 2 r J'-r k_1+ a^1
P'+P^{(3)}\right) - \left(3a^0  J' +a^1  J'' + J^{(4)} \right)
\end{align}
\end{subequations}
and where $J$ and $P$ are arbitrary functions of $x$ and $k_1$ is an
arbitrary constant, while $\phi^4$ is an arbitrary function of $x,y,
a^1, a^0$ and $r.$ Thus $X$ has projection $V= \set{J, (k_1+ J')y +
P}$ on $(x,y)$-space and this is exactly the infinitesimal
transformation of \eqref{eq:eqv3nh}. Consequently, the minimum set
$\Omega$ of conditions to be imposed on $V$ to reduce it into the
infinitesimal generator $V^0=\set{\xi^0, \eta^0}$ of $G$ is void in
this case, and hence
\begin{equation}\label{eq:x=x1}
X= X^1= \set{\xi^0, \eta^0, \phi^0}.
\end{equation}
It thus follows from Lemma \ref{le:phi0} that the finite
transformations associated  with $X$ are given precisely by
\eqref{eq:eqv3nh}, together with the  corresponding induced
transformations of the arbitrary functions $a^1, a^0$ and $r.$
Consequently, to each symmetry transformation $X$ in $G_S,$ there
corresponds a  unique equivalence transformation in $G,$ and vice-versa. We
have thus proved the following results.
\begin{pro}
For the nonhomogeneous equation \eqref{eq:e3nh}, the groups $G$ and
$G_S$ are isomorphic.
\end{pro}
  This proposition should certainly also hold for the nonhomogeneous
version of the general linear equation \eqref{eq:glinode} of an
arbitrary order $n.$ In such cases, invariants  of the differential
equation are determined  simply by searching the symmetry generator
$X$ of $G_S,$ which must satisfy \eqref{eq:x=x1}, and then solving
the resulting system of linear first order partial differential
equations (PDEs) resulting from the determining equation of the form
\begin{equation}\label{eq:deteq1}
X^{0, m} \cdot F=0,
\end{equation}
where $X^{0,m}$ is the generator $X^0= \set{\xi^0, \phi^0}$ of $G_c$
prolonged to the desired order $m$ of the unknown invariants
$F.$\par
\section*{Concluding Remarks}

  Because Type I equivalence group $G$ can be identified with a
  subgroup of type II equivalence group $G_S,$ every function
  invariant under $G_S$ must be invariant under $G,$ and hence $G$
  has much more invariant functions than $G_S,$ and functions
  invariant under $G$ are naturally much easier to find than those
  invariant under $G_S.$ If we consider for instance the third order
linear equation  \eqref{eq:e3nor}, it is well known \cite{forsyth}
that its first nontrivial invariant function  is given by the third
order differential invariant
\begin{equation}\label{eq:inve3nor3}
\Psi= - \frac{4(9 a^1 \mu^2 + 7 \mu'^2 - 6 \mu \mu'')^3}{\mu^8},
\end{equation}
where $\mu(x)= -2a^0+ {a^1}',$ while at order four \cite{ndogpla} it
has two differential invariants,
\begin{align*}
\Psi_1 &= \Psi \\
\begin{split} \Psi_2 &=  \frac{-1}{18 \mu^4} \left( 216 a^{\!0 ^{\,4}} - 324 a^{\! 0^{\,3}}\, {a^1}' + 18 \, a^{\!0^{\,2}} (9 {a^1}'{\,^2}+ 2 a^1
\mu') + + 9 \mu^2 \mu^{(3)}\right) +\\
 & \quad \frac{-1}{18 \mu^4} \left(\mu' (28 \mu'\,^2 + 9 {a^1}'(a^1\, {a^1}' - 4 \mu'')) - 9 a^0(3 {a^1}'^{\,3}+ 4 a^1 \, {a^1}' \mu'- 8 \mu' \mu'') \right).\end{split}
\end{align*}

It can be verified on the other hand that $G_S$ has no nontrivial
differential invariants up to the order four.





\end{document}